\documentclass[a4paper,12pt]{article}
\usepackage{amssymb}
\usepackage{epsfig}
\usepackage{amsfonts}
\usepackage{amsmath}
\usepackage{euscript}
\usepackage{amscd}
\usepackage{amsthm}
\DeclareMathAlphabet{\mathpzc}{OT1}{pzc}{m}{it}
\newtheorem{thm}{Theorem}[section]

\newtheorem{rem}[thm]{Remark}

\newcommand{\p}{\mathpzc{p}}

\newcommand{\bQ}{\mathbb Q}

\title{A note on partial coordinate system in a polynomial ring}
\author{Animesh Lahiri\\
{\small{\it Swami Vivekananda Research Centre}}\\
{\small{\it Ramakrishna Mission Vidyamandira}}\\
{\small{\it Belur Math, Howrah, India}}\\
{\small{\it e-mail: {255alahiri@gmail.com}}
}}
\begin{document}
\date{}
\maketitle
\abstract
\noindent 
In the paper \cite{BBE}, J. Berson, J. W. Bikker and A. van den Essen proved that for a non-zerodivisor $a$ in a commutative ring $R$ containing $\bQ$ if the polynomials $f_1,\dots,f_{n-1}$ in $R[X_1,\dots,X_n]$ form a partial coordinate system over the rings $R_a$ and $\dfrac{R}{aR}$ then $f_1,\dots,f_{n-1}$ form a partial coordinate system over the ring $R$. In this note we show that the theory of residual variables of Bhatwadekar-Dutta (\cite{BDr}) and its recent extension by Das-Dutta (\cite{DD}), extends their result to the case when $a$ is an arbitrary element of $A$. 
\smallskip
  
\noindent
{\small {{\bf Keywords.} Polynomial algebra, Coordinate, Residual coordinate.}}

\noindent
{\small {{\bf AMS Subject classifications (2010)}. Primary: 13B25; Secondary: 14R25}}.
\section{Introduction}
We will assume all rings to be commutative containing unity. Throughout this note we will use the notation $R[X_1,\dots,X_n]$ to mean a polynomial algebra
in $n$ variables over $R$. Sometimes we will denote this ring by $R^{[n]}$.

\medskip
\noindent
Let $A:=R[X_1,\dots,X_n]$. We recall that $m$ polynomials $f_1,\dots,f_m$ ($m\leq{n}$) in $A$ are said to form a partial coordinate system (of colength $n-m$) in $A$ if $A= R[f_1,\dots,f_m]^{[n-m]}$. If $m=n$ then we will say $f_1,\dots,f_n$ form a coordinate system in $A$. For an arbitrary $a\in{R}$, $f_1,\dots,f_m$ in $A$ are said to form an $a$-strongly partial residual coordinate system (of colength $n-m$) in $A$ if the images of $f_1,\dots,f_m$ form a partial coordinate system (of colength $n-m$) in $\dfrac{A}{aA}$ and also in $A_a$. We will say $f_1,\dots,f_m$ form a partial residual coordinate system (of colength $n-m$) in $A$  if, for each prime ideal
$\p$ of $R$ we have $k(\p)\otimes_R{A}=(k(\p)\otimes_{R}{R[f_1,\dots,f_m]})^{[n-m]}$, where $k(\p):=\dfrac{R_{\p}}{\p R_{\p}}$ is the residue field of $R$ at $\p$.

\medskip
\noindent
The following result has been proved by J. Berson, J. W. Bikker and A. van den Essen in \cite{BBE}.
\begin{thm}\label{essen}
Let $R$ be a ring containing $\bQ$, $a\in{R}$ a non-zerodivisor and $A=R[X_1,\dots,X_n]$. If $n-1$ polynomials $f_1,\dots,f_{n-1}$ in $A$
form an $a$-strongly partial residual coordinate system in $A$, then $f_1,\dots,f_{n-1}$ form a partial coordinate system in $A$.
\end{thm}
They conjectured the following (\cite[Conjecture 4.4]{BBE}):

\medskip
\noindent
{\bf Conjecture.} 
Theorem \ref{essen} also holds even when $a$ is a zerodivisor in $R$.

\medskip
In this note we observe that an affirmative solution to the above conjecture can be deduced from the following formulation of a result of Das-Dutta (\cite[Corollary 3.19]{DD}).
\begin{thm}\label{colength 1}
Let $R$ be a Noetherian ring containing $\bQ$, $A=R^{[n]}$ and $f_1,\dots,f_{n-1}\in{A}$. Then the following are equivalent:
\begin{enumerate}
\item[\rm (i)] $f_1,\dots,f_{n-1}$ form a partial coordinate system in $A$.
\item[\rm (ii)] $f_1,\dots,f_{n-1}$ form a partial residual coordinate system in $A$.
\end{enumerate} 
\end{thm}
\begin{rem}
{\em We note the following observations on the above result

(i) Although the result  is stated in the paper \cite{DD} for Noetherian domains containing $\bQ$, the proof is an application of Theorem 3.13 and Theorem 2.4 in \cite{DD} 
which hold over any Noetherain ring containing $\bQ$ (not necessarily a domain).
  
(ii) The case $n=2$ was previously proved by Bhatwadekar-Dutta (\cite[Theorem 3.1]{BDr}). 
}
\end{rem}
\section{Proof of the conjecture}
\begin{thm}
Let $R$ be a ring containing $\bQ$, $a\in{R}$ be arbitrary and $A=R[X_1,\dots,X_n]$. If  $n-1$ polynomials $f_1,\dots,f_{n-1}$ in $A$
form an $a$-strongly partial residual coordinate system of colength $1$ in $A$, then $f_1,\dots,f_{n-1}$ form a partial coordinate system in $A$.
\end{thm}
\begin{proof}
Since $f_1,\dots,f_{n-1}$ form an $a$-strongly partial residual coordinate system in $A$, 
there exist $g, h\in{A}$ such that
$$
\dfrac{R}{aR}[X_1,\dots,X_{n}]=\dfrac{R}{aR}[\overline{f_1},\dots,\overline{f_{n-1}},\overline{g}]
$$
and 
$$
R_a[X_1,\dots,X_{n}]=R_a[f_1,\dots,f_{n-1},h].
$$
Hence we have, 
\begin{equation}\label{mod}
X_i = G_i + a H_i
\end{equation}

for some $G_i\in{R[f_1,\dots,f_{n-1},g]}$ and $H_i\in{A}$ $(1\leq{i}\leq{n})$;\\
  
 and
 \begin{equation}\label{local}
  X_i = \sum{\dfrac{e_{i,m_1,\dots,m_{n-1},m}}{a^{k_i}}{f_1^{m_1}f_2^{m_2}\dots f_{n-1}^{m_{n-1}}h^m}}
 \end{equation}
 where $e_{i,m_1,\dots,m_{n-1},m}\in{R}$, for all ${i,m_j,m}$  $(1\leq{i}\leq{n}, 1\leq{j}\leq{n-1},{m_j,m,k_i}\geq{0})$.

Now, let $S$ be the $\bQ$-subalgebra of 
$R$ generated by the subset  of $R$ consisting of $a$; all coefficients of $G_i$ (for all $i$)  as polynomials in $f_1,\dots,f_{n-1},g$; coefficients
of $H_i$ (for all $i$); coefficients of $g, h, f_i$ (for all $i$) and $e_{i,m_1,\dots,m_{n-1},m}$ (for all $i$ and for all ${m,m_j})$.
Then $S$ is a Noetherian $\bQ$-algebra. Let  $\p$ be an arbitrary prime ideal of $S$ and $B:=S[X_1,\dots,X_n]$.
If $a\in{\p}$ then from (\ref{mod}), we have $k(\p)\otimes_S{B}=(k(\p)\otimes_{S}{S[f_1,\dots,f_{n-1}]})^{[1]}$.
If $a\notin{\p}$ then from (\ref{local}), we have $k(\p)\otimes_S{B}=(k(\p)\otimes_{S}{S[f_1,\dots,f_{n-1}]})^{[1]}$. 
So, $f_1,\dots,f_{n-1}$ form a partial residual coordinate system (of colength $1$) in $B$. Since $S$ is Noetherian, hence by Theorem \ref{colength 1}, $f_1,\dots,f_{n-1}$ form a partial coordinate system  in $B$ and hence in $A$.
\end{proof}

\bigskip
\noindent
{\bf Acknowledgements.} The author acknowledges Council of Scientific and Industrial Research (CSIR) for their research grant.  

{\small{

}}
\end{document}